\documentclass[a4paper,10pt]{article}
\usepackage{authblk}

\usepackage{cancel}
\usepackage{pgfplots}
\usepackage{url}

\usepackage[mathlines]{lineno}

\usepackage{enumitem}

\usepackage[all]{xy}
\usepackage[english]{babel}
\usepackage[margin=1in]{geometry}
\usepackage[utf8]{inputenc}
\usepackage{amsmath}
\usepackage{mathtools}
\usepackage{amsthm}
\usepackage{times}
\usepackage{hyperref}
\usepackage{pifont}
\usepackage{booktabs}  
\usepackage{multirow}
\usepackage{blkarray}

\usepackage{amsfonts}
\usepackage{graphicx}
\usepackage{amssymb}
\usepackage{arydshln}
\usepackage{tikz}
\usepackage{tikz-cd}
\usepackage{bbold}
\usepackage[percent]{overpic}

\newtheorem{theorem}{Theorem}
\newtheorem*{theorem*}{Theorem}
\newtheorem*{proposition*}{Proposition}

\newtheorem{proposition}{Proposition}[section]
\newtheorem{lemma}[proposition]{Lemma}
\newtheorem{remark}[proposition]{Remark}

\newtheorem{problem}[proposition]{Problem}
\newtheorem{definition}[proposition]{Definition}

\newtheorem{example}[proposition]{Example}
\newtheorem{examples}[proposition]{Examples}

\newcommand{\mX}{\mathcal{X}}

\newcommand{\mIX}{\mathcal{I}_\mathcal{X}}
\newcommand{\mIY}{\mathcal{I}_\mathcal{Y}}

\newcommand{\mY}{\mathcal{Y}}
\newcommand{\mZ}{\mathcal{Z}}

\newcommand{\mT}{\mathcal{T}}

\newcommand{\R}{\mathbb{R}}

\newcommand{\catC}{\mathsf{C}}

\newcommand{\mnets}{\mathsf{N}}
\newcommand{\mmspaces}{\mathsf{M}}

\newcommand{\define}[1]{\textbf{#1}}

\newcommand{\dgw}{\mathrm{GW}_p}

\newcommand{\dgm}{\mathrm{GM}_p}

\newcommand{\dst}{\mathrm{GW}^{\mathrm{em}}_p}
\newcommand{\dstm}{\mathrm{GM}^{\mathrm{em}}_p}

\newcommand{\dw}{\mathrm{W}_p}

\newcommand{\dm}{\mathrm{M}_p}

\definecolor{darkblue}{rgb}{0.0, 0.0, 0.8}
\definecolor{darkred}{rgb}{0.8, 0.0, 0.0}
\definecolor{darkgreen}{rgb}{0.0, 0.8, 0.0}

\newcommand{\supp}{\mathrm{supp}}

\usepackage{mathtools}

\newcommand{\Expect}{{\rm I\kern-.3em E}}

\title{Comparison Results for Gromov-Wasserstein and Gromov-Monge Distances}

\author[1]{Facundo M\'emoli}
\author[2]{Tom Needham}

\affil[1]{Department of Mathematics,
	The Ohio State University\\ 
	\texttt{facundo.memoli@gmail.com}}
\affil[2]{Department of Mathematics,
	Florida State University\\
	\texttt{tneedham@fsu.edu}}
    
\date{}

\pgfplotsset{compat=1.14}

\let\oldequation\equation
\let\oldendequation\endequation

\renewenvironment{equation}
  {\linenomathNonumbers\oldequation}
  {\oldendequation\endlinenomath}
  
\begin{document}
\maketitle
\sloppy

\begin{abstract}

Inspired by the Kantorovich formulation of optimal transport distance between probability measures on a metric space, Gromov-Wasserstein (GW) distances comprise a family of  metrics on the space of isomorphism classes of metric measure spaces. In previous work, the authors introduced a variant of this construction which was inspired by the original Monge formulation of optimal transport; elements of the resulting family are referred to Gromov-Monge (GM) distances. These GM distances, and related ideas, have since become a subject of interest from both theoretical and applications-oriented perspectives. In this note, we establish several theoretical properties of GM distances, focusing on comparisons between GM and GW distances. In particular, we show that GM and GW distances are equal for non-atomic metric measure spaces. We also consider variants of GM distance, such as a Monge version of Sturm's $L_p$-transportion distance, and give precise comparisons to GW distance. Finally, we establish bi-H\"{o}lder equivalence between GM distance and an isometry-invariant Monge optimal transport distance between Euclidean metric measure spaces that has been utilized in shape and image analysis applications.

\end{abstract}

\section{Introduction}

Gromov-Wasserstein (GW) distance is a metric which allows one to compare probability measures defined over \emph{different} metric spaces. This is a task which is necessary in many applications in shape analysis and machine learning, and GW distance has consequently become popular in these fields; see \cite{peyre2016gromov,alvarez2018gromov,bunne2019learning,xu2019gromov,chowdhury2021quantized,demetci2022scot}, among many others. A primary reason for the usefulness of this framework is that computation of GW distance involves finding a soft correspondence between points of the two metric spaces being compared (more precisely, a probability coupling; see below), which gives a meaningful registration of the spaces for downstream analysis. On the other hand, it is sometimes desirable to require an exact point registration (i.e., a function) between the spaces, and this led us to define a variant of GW distance, called Gromov-Monge (GM) distance, in our previous paper \cite{memoli2022distance}. Roughly, GM defines an optimization problem with the same objective as GW, but restricts the feasible set from soft correspondences to exact correspondences, so that GW $\leq$ GM, in general. The GM distance, and related ideas, have since appeared in both theoretical and applied contexts \cite{vayer2020contribution, salmona2022gromov,dumont2024existence,hur2024reversible,zhang2022cycle} (see Section \ref{sec:literature} for a detailed literature review). The purpose of the present paper is to collect some results on GM distance which have not appeared elsewhere\footnote{Some of these results appeared in early versions of the arXiv version of our paper \cite{memoli2022distance}, but were removed from the published version for the sake of brevity.}. The overarching theme is to address the following \textbf{main question}: what are conditions on the metric spaces/measures being compared which imply that their GW and GM distances agree?

In order to describe our results, we now introduce some precise definitions. A \emph{metric measure space (mm-space)} is a triple $\mX = (X,d_X,\mu_X)$ such that $(X,d_X)$ is a metric space and $\mu_X$ is a fully supported Borel probability measure on $X$. For $p \in [1,\infty)$, the \emph{Gromov-Wasserstein $p$-distance} between two such structures $\mX$ and $\mY$ is given by
\[
\mathrm{GW}_p(\mX,\mY) \coloneqq \inf_{\pi}  \left(\iint_{(X \times Y)^2} |d_X(x,x') - d_Y(y,y')|^p \pi (dx \times dy) \pi(dx' \times dy')\right)^\frac{1}{p},
\]
where the infimum is over \emph{couplings} of $\mu_X$ and $\mu_Y$; that is, joint probability measures $\pi$ on $X \times Y$ whose marginals are $\mu_X$ and $\mu_Y$. This defines a metric on the space of mm-spaces up to a suitable notion of isomorphism, and extends in a natural way to $p=\infty$. As it is defined as an optimization problem over the space of couplings, GW distance is reminiscent of Kantorovich's formulation of optimal transport, which gives rise to Wasserstein $p$-distances (see \cite{villani2008optimal,peyre2019computational} for general references on optimal transport theory). If we instead extend Monge's original formulation of optimal transport, we are led to the notion of \emph{Gromov-Monge $p$-distance},
\[
\mathrm{GM}_p(\mX,\mY) \coloneqq \inf_{\phi}  \left(\iint_{X \times X} |d_X(x,x') - d_Y(\phi(x),\phi(x'))|^p \mu_X(dx) \mu_X(dx')\right)^\frac{1}{p},
\]
where the infimum is now over functions $\phi:X \to Y$ which push the measure $\mu_X$ forward to $\mu_Y$. It is straightforward to show that $\mathrm{GW}_p(\mX,\mY) \leq \mathrm{GM}_p(\mX,\mY)$, in general (Remark \ref{rmk:dgw_dgm_inequality}). 

The results of this paper explore connections between GW and GM distances (and related constructions) under various assumptions on the spaces being compared. In fact, we work in the more general setting of \emph{measure networks}; that is, triples $\mX = (X,\omega_X,\mu_X)$ where the \emph{network function} $\omega_X:X \times X \to \R$ is an arbitrary measurable map (Definition \ref{def:measure_network}). Our main contributions are:
\begin{itemize}
\item We show that $\dgm(\mX,\mY) = \dgw(\mX,\mY)$ for measure networks with non-atomic probability measures (Theorem \ref{thm:nonatomic_equivalence}). This generalizes a result in the recent paper \cite{hur2021reversible} and is analogous to results which show the equivalence of Kantorovich and Monge (classical) optimal transport for non-atomic measures \cite{ambrosio2003lecture,pratelli2007equality}. This is also related to a growing body of work which seeks to characterize classes of measure networks where $\dgw$ is realized by a Monge map \cite{sturm2023space, vayer2020contribution, salmona2022gromov}; in particular, recent work of Dumont, Lacombe and Vialard gives a complete solution for certain measure networks over Euclidean spaces \cite{dumont2024existence}. This line of work is inspired by a famous result of Brenier for classical optimal transport \cite{brenier1991polar}. 
\item For measure networks over finite, uniformly distributed probability spaces with symmetric positive definite network functions, we show $\mathrm{GM}_2 = \mathrm{GW}_2$ (Proposition \ref{prop:positive_definite}); in particular, $\mathrm{GW}_2$ is always realized by a Monge map in this setting. This has applications to, e.g., graph analysis, where $\omega$ can be the heat kernel. Our result is in the spirit of a result in \cite{vayer2020contribution}, which applies to conditionally negative definite network functions.
\item In the mm-space setting, we show that $\dgw$ is equal to a variant of $\dgm$ which allows for a certain notion of \emph{mass splitting} (Theorem \ref{thm:monge_wasserstein_comparison}). This refines various observations about GW distance that have appeared previously \cite{sturm2023space,chowdhury2019gromov,chowdhury2022distances}.
\item We define an alternate version of GM distance for mm-spaces based on embeddings into a common ambient space, in analogy with Sturm's $L_p$-transportation distance~\cite{sturm2006geometry}, and show that the embedding GM distance and Sturm's distance coincide for non-atomic mm-spaces (Theorem \ref{thm:gm_equals_gw_embedding}). Moreover, we show that the GM (as defined above) gives a lower bound on the embedding version, in general (Theorem \ref{thm:gromov_monge_reformulation}). 
\item Previous works in shape and image analysis consider isometry-invariant versions of Monge optimal transport \cite{haker2004optimal,boyer2011algorithms,al2013continuous}. We show that isometry-invariant Monge distance for Euclidean mm-spaces is bi-H\"{o}lder equivalent to the embedding formulation of GM distance (Theorem \ref{thm:gromov_monge_Euclidean}).
\end{itemize}

The structure of the paper is as follows. In Section \ref{sec:gromov_monge_distances}, we give the necessary background definitions and survey recent literature. Section \ref{sec:gw_as_gm} presents results on equivalence of GW and GM distances for certain classes of measure networks. Finally, in Section \ref{sec:gromovization},  we introduce the embedding formulation of GM distance, give comparison results with the corresponding GW version and show bi-H\"{o}lder equivalence with isometry-invariant Monge distances for Euclidean spaces.

\section{Gromov-Wasserstein and Gromov-Monge Distances}\label{sec:gromov_monge_distances}

This section introduces our basic notation and gives precise definitions of various optimal transport-based distances.

\subsection{Measure Networks and Metric Measure Spaces}

As was described above, the original work in Gromov-Wasserstein (GW) distances \cite{memoli2007use} concerned comparison of mm-spaces, but the purview of the theory has since been extended \cite{chowdhury2019gromovNetworks, titouan2020co, chowdhury2023hypergraph, chen2022weisfeiler}. In this paper, we will work in the following setting, borrowing terminology from \cite{chowdhury2019gromovNetworks}.

\begin{definition}\label{def:measure_network}
A \define{measure network} (\define{m-net}) is a triple $\mathcal{X} = (X,\omega_X,\mu_X)$ such that $X$ is a separable and completely metrizable topological space (or \define{Polish space}), $\mu_X$ is a fully-supported Borel probability measure and $\omega_X:X \times X \to \R$ is a measurable function, which we refer to as the \define{network function}. We use $\mnets$ to denote the class of measure networks. 

In the case that the network function is a metric on $X$, we will typically denote it as $d_X$, we will assume that $d_X$ generates the topology of $X$ and we will refer to the triple $\mathcal{X} = (X,d_X,\mu_X)$ as a \define{metric measure space (mm-space)}. We denote the class of mm-spaces by $\mmspaces$.
\end{definition}

\begin{examples}\label{ex:mm-spaces}
There are many useful examples of m-nets which are not mm-spaces. Here are a few:
\begin{itemize}
\item It is frequently useful to consider network functions which satisfy a subset of the metric axioms. For instance, a \define{pseudo-metric} is a function $\omega_X:X \times X \to \R$ satisfying all of the metric axioms, besides allowing $\omega_X(x,x') = 0 $ for $x \neq x'$. Pseudo-metric measure spaces arise naturally as elements of the completion of the space of metric measure spaces with respect to Gromov-Wasserstein distances \cite{sturm2023space}.
\item A graph $G = (V,E)$ with (finite) vertex set $V$ and edge set $E$ can be represented as an m-net by setting $X = V$, taking $\mu_X$ to be uniform and defining $\omega(x,x') \in \{0,1\}$ to be $1$ if and only if $\{x,x'\} \in E$  \cite{xu2019scalable}.
\item The previous example can be generalized by taking $\omega$ to be some other graph kernel, such as the graph Laplacian or a heat kernel \cite{chowdhury2021generalized}. 
\end{itemize}
\end{examples}

There are two notions of equivalence of measure networks which are relevant to the GW framework. We recall that a measurable function $\phi:X \to Y$ between measure spaces $(X,\mu_X)$ and $(Y,\mu_Y)$ is \define{measure-preserving} if $\phi_\# \mu_X = \mu_Y$, where $\phi_\# \mu_X (A) \coloneqq \mu_X(\phi^{-1}(A))$, for any measurable $A \subset Y$. 

\begin{definition}
Let $\mX$ and $\mY$ be m-nets. A \define{strong isomorphism} from $\mX$ to $\mY$ is a measure-preserving bijective map $\phi:X \to Y$ with measure-preserving inverse, such that $\omega_Y(\phi(x),\phi(x')) = \omega_X(x,x')$ for all $(x,x') \in X \times X$. 
\end{definition}

\begin{definition}\label{def:weak_isomorphism}
Measure networks $\mX$ and $\mY$ are \define{weakly isomorphic} if there exists a Borel probability space $(Z,\mu_Z)$, together with measure-preserving maps $\phi_X:Z \to X$ and $\phi_Y:Z \to Y$ such that $\omega_X(\phi_X(z),\phi_X(z')) = \omega_Y(\phi_Y(z),\phi_Y(z'))$ holds for $\mu_Z \otimes \mu_Z$-almost every $(z,z') \in Z \times Z$. 
\end{definition}

Trivially, if $\mX$ and $\mY$ are strongly isomorphic then they are weakly isomorphic. The converse holds if $\mX$ and $\mY$ are mm-spaces, but not in general (see Theorem \ref{thm:summary} below).

\subsection{Distances Between Measures and Measure Networks}

Optimal transport distances are defined in terms of certain joint probability measures called couplings, which we now recall.

\begin{definition}\label{def:couplings}
Let $(X,\mu_X)$ and $(Y,\mu_Y)$ be probability spaces. A \define{coupling} of $\mu_X$ and $\mu_Y$ is a probability measure $\pi$ on $X \times Y$ whose marginals are $\mu_X$ and $\mu_Y$, respectively. That is, let $\rho_X:X \times Y \to X$ and $\rho_Y:X \times Y \to Y$ be coordinate projections; then a coupling satisfies $(\rho_X)_\# \pi = \mu_X$ and $(\rho_Y)_\# \pi = \mu_Y$. The collection of all couplings of $\mu_X$ and $\mu_Y$ will be denoted $\mathcal{C}(\mu_X,\mu_Y)$. 

For a measure-preserving map $\phi:X \to Y$, let $\mathrm{id}_X \times \phi: X \to X \times Y$ be defined by $(\mathrm{id}_X \times \phi) (x,x') = (x,\phi(x'))$. We define the \define{coupling induced by $\phi$} to be
$\pi_\phi \coloneqq (\mathrm{id}_X \times \phi)_\# \mu_X \in \mathcal{C}(\mu_X,\mu_Y)$. Let $\mathcal{T}(\mu_X,\mu_Y)$ denote the collection of all measure-preserving mappings from $(X,\mu_X)$ to $(Y,\mu_Y)$ and let 
\[
\mathcal{T}_\#(\mu_X,\mu_Y) \coloneqq \{\pi_\phi \mid \phi \in \mathcal{T}(\mu_X,\mu_Y)\} \subset \mathcal{C}(\mu_X,\mu_Y)
\]
denote the set of couplings induced by measure-preserving mappings.
\end{definition}

\begin{remark}
The set $\mathcal{C}(\mu_X,\mu_Y)$ is always nonempty, since it contains the product measure $\mu_X \otimes \mu_Y$. However, the set $\mathcal{T}(\mu_X,\mu_Y)$ (or $\mathcal{T}_\#(\mu_X,\mu_Y)$) can be empty. For example, if $\mu_X$ is a Dirac measure on a single point and $\mu_Y$ is a sum of Dirac measures $\frac{1}{2}(\delta_{y} + \delta_{y'})$ for $y \neq y'$, then there is no measure-preserving map $\phi:X \to Y$. 
\end{remark}

With the notion of coupling in hand, we can now define some classical optimal transport distances between measures over the same metric space. We do so in the setting of metric spaces, since this level of generality will be sufficient for our purposes. For background on classical optimal transport, see, e.g.,~\cite{villani2008optimal}. In the following, we allow our distances to take the value $\infty$; that is, they may be \define{extended metrics}.

\begin{definition}
Let $(X,d_X)$ be a metric space and let $\mu$ and $\nu$ be Borel probability measures on $X$. The \define{Wasserstein $p$-distance} is defined for $p \in [1,\infty)$ by
\[
\mathrm{W}_p(\mu,\nu)^p \coloneqq \inf_{\pi \in \mathcal{C}(\mu,\nu)} \int_{X \times X} d_X(x,x')^p \pi(dx \times dx') = \inf_{\pi \in \mathcal{C}(\mu,\nu)} \|d_X\|_{L^p(\pi)}^p
\]
and for $p = \infty$ by 
\[
\mathrm{W}_\infty(\mu,\nu) \coloneqq \inf_{\pi \in \mathcal{C}(\mu,\nu)} \sup \{d_X(x,x') \mid (x,x') \in \mathrm{supp}(\pi)\} = \inf_{\pi \in \mathcal{C}(\mu,\nu)} \|d_X\|_{L^\infty(\mathrm{supp}(\pi))},
\]
where $\mathrm{supp}(\pi)$ is the support of $\pi$.

We likewise define the \define{Monge $p$-distance} by restricting the feasible set of couplings; that is,
\[
\mathrm{M}_p(\mu,\nu) \coloneqq \left\{\begin{array}{cl}
\inf_{\pi \in \mathcal{T}_\#(\mu,\nu)} \|d_X\|_{L^p(\pi)} & p < \infty \\
\inf_{\pi \in \mathcal{T}_\#(\mu,\nu)} \|d_X\|_{L^\infty(\mathrm{supp}(\pi))} & p = \infty,
\end{array}\right.
\]
where we take the infimum over the empty set to be infinity. 

When the underlying metric space needs to emphasized, we will write $\dw^X = \dw$ and $\dm^X = \dm$.
\end{definition}

These distances have analogues which measure the distance between pairs of mm-spaces or measure networks. Since the underlying spaces are allowed to vary, we need to measure the quality of a coupling in a relative manner, leading to the following definition.

\begin{definition}\label{def:distortion}
Let $\mX$ and $\mY$ be measure networks and $p \in [1,\infty)$. The \define{$p$-distortion} of a coupling $\pi \in \mathcal{C}(\mu_X,\mu_Y)$ is 
\begin{align*}
\mathrm{dis}_p(\pi)^p = \mathrm{dis}_{p,\mX,\mY}(\pi)^p &\coloneqq  \iint_{(X \times Y)^2} \left| \omega_X(x,x') - \omega_Y(y,y') \right|^p \pi(dx \times dy) \pi(dx' \times dy') \\
&= \|\omega_X - \omega_Y\|_{L^p(\pi \otimes \pi)}^p.
\end{align*}
Similarly, for $p = \infty$, we define
\[
\mathrm{dis}_\infty(\pi) = \mathrm{dis}_{\infty,\mX,\mY}(\pi) \coloneqq \sup \{ \left| \omega_X(x,x') - \omega_Y(y,y') \right| \mid (x,y),(x',y') \in \mathrm{supp}(\pi)\}.
\]

We define the \define{$p$-distortion} of a measure-preserving map $\phi \in \mathcal{T}(\mu_X,\mu_Y)$ to be
$
\mathrm{dis}_p(\phi) = \mathrm{dis}_p(\pi_\phi),
$ 
with $\pi_\phi$ as in Definition \ref{def:couplings}. The expression for the distortion simplifies in the $p < \infty$ case as
\[
\mathrm{dis}_p(\phi)^p = \iint_{X \times X} \left| \omega_X(x,x') - \omega_Y(\phi(x),\phi(x')) \right|^p \mu_X(dx) \mu_X(dx'),
\]
and similarly in the $p=\infty$ case.
\end{definition}

Finally, we define our distances between measure networks. These agree in the mm-space setting with the definitions given in the introduction.

\begin{definition}
For $p \in [1,\infty]$, the \define{Gromov-Wasserstein $p$-distance} is the function $\dgw: \mnets \times \mnets \to \R\cup \{\infty\}$ defined by
\[
\dgw(\mX,\mY) \coloneqq \inf_{\pi \in \mathcal{C}(\mu_X,\mu_Y)} \mathrm{dis}_p(\pi).
\]
The \define{Gromov-Monge $p$-distance} is the function $\dgm: \mnets \times \mnets \to \R \cup \{\infty\}$ defined by
\[
\dgm(\mX,\mY) \coloneqq \inf_{\phi \in \mathcal{T}(\mu_X,\mu_Y)} \mathrm{dis}_p(\phi) = \inf_{\pi_\phi \in \mathcal{T}_\#(\mu_X,\mu_Y)} \mathrm{dis}_p(\pi_\phi),
\]
where we once again declare the infimum over the empty set to be infinity. 
\end{definition}

\begin{remark}\label{rmk:dgw_dgm_inequality}
Since $\dgw$ involves optimization over a larger set of couplings than $\dgm$, the inequality $\dgw(\mX,\mY) \leq \dgm(\mX,\mY)$ always holds.
\end{remark}

\begin{remark}\label{rmk:optimal_couplings}
The infimum in the definition of $\dgw$ is actually a minimum. This is proved in~\cite[Theorem 2.2]{chowdhury2019gromovNetworks} for m-nets whose network functions are bounded and whose measures are fully supported, and in~\cite[Lemma 1.7]{sturm2023space} for m-nets whose network functions are metrics, i.e., for mm-spaces (without any boundedness or support assumptions). Working through the proofs, one sees that the various assumptions are not necessary, and the result goes through in the general setting of Definition \ref{def:measure_network}.
\end{remark}

In this paper, we consider m-nets with potentially unbounded network functions, meaning that $\dgw$ is not guaranteed, in general, to be finite. It is sometimes useful to restrict to subspaces of m-nets with additional control on the network functions. It is straightforward to see that $\dgw$ is finite when restricted to $\mnets_p \times \mnets_p$, where $\mnets_p$ is as defined below.

\begin{definition}
    For $p \in [1,\infty)$ and an m-net $\mX$, define the \define{$p$-size} of $\mX$, $\mathrm{size}_p(\mX) \in \R \cup \{\infty\}$, by
    \[
    \mathrm{size}_p(\mX)^p \coloneqq \iint_{X \times X} |\omega_X(x,x')|^p \mu_X(dx) \mu_X(dx').
    \]
    Similarly, let 
    \[
    \mathrm{size}_\infty(\mX) \coloneqq \sup_{x,x' \in X} |\omega_X(x,x')|.
    \]
    We denote the class of m-nets (respectively, mm-spaces) $\mX$ with $\mathrm{size}_p(\mX) < \infty$ by $\mnets_p$ (respectively, $\mmspaces_p$).
\end{definition}

\begin{remark}\label{rem:p-size}
    The terminology above is used in, e.g., \cite{sturm2023space}. In the setting of mm-spaces, $\mathrm{size}_p(\mX)$ is sometimes referred to as the \define{$p$-diameter} of $\mX$. It appears in certain estimates of GW distance---see~\cite[Theorem 5.1]{memoli2011gromov}. In particular, letting $\ast$ denote the 1-point mm-space,
    \[
    \dgw(\mX,\ast)  = \mathrm{size}_p(\mX) = \dgm(\mX,\ast),
    \]
    where the first equality appears in~\cite[Theorem 5.1 (f)]{memoli2011gromov} and the second follows because the unique coupling of $\mX$ and $\ast$ is induced by a measure-preserving map.
\end{remark}

The following theorem summarizes some basic results on metric properties of GM and GW distances in the literature. We recall some terminology from metric geometry. For a set $Z$, a function $d: Z \times Z \to \R \cup \{\infty\}$ is a \define{Lawvere metric} if it satisfies the triangle inequality and $d(z,z) = 0$ for all $z \in Z$---this terminology refers to classic work of Lawvere, which characterizes such structures in the language of enriched category theory~\cite{lawvere1973metric}. A Lawvere metric is a \define{pseudometric} if it is, in addition, finite-valued and symmetric (i.e., it satisfies the axioms of a metric, except it is possible that $d(z,z') = 0$ for $z \neq z'$), as in Example \ref{ex:mm-spaces}. 

\begin{theorem}\label{thm:summary}
Let $p \in [1,\infty]$.
\begin{itemize}
\item On the space $\mmspaces_p$, the Gromov-Wasserstein $p$-distance defines a pseudometric such that $\dgw(\mX,\mY) = 0$ if and only if $\mX$ and $\mY$ are strongly isomorphic \cite{memoli2007use,sturm2023space}.
\item On the space $\mmspaces$, the Gromov-Monge $p$-distance defines a Lawvere metric such that $\dgm(\mX,\mY) = 0$ if and only if $\mX$ and $\mY$ are strongly isomorphic. It may take the value $\infty$ even on $\mmspaces_p$  \cite{memoli2022distance}.
\item On the space $\mnets_p$, the Gromov-Wasserstein $p$-distance defines a pseudometric such that $\dgw(\mX,\mY) = 0$ if and only if $\mX$ and $\mY$ are weakly isomorphic \cite{chowdhury2019gromovNetworks}.
\end{itemize}
\end{theorem}

\begin{remark}
    Let us comment on some subtleties of the results referenced in Theorem \ref{thm:summary}. When $\dgw$ was first introduced, it was shown to define a metric up to strong isomorphism on the space of \emph{compact} mm-spaces~\cite[Proposition 6]{memoli2007use}. This was later extended to $\mmspaces_p$ in \cite[Proposition 1.12]{sturm2023space}; the proof is essentially the same as in the compact case. The case of $\dgw$ for m-nets was formalized in \cite{chowdhury2019gromovNetworks}, where it was shown to define a metric on $\mnets_\infty$ (m-nets with bounded network functions), considered up to weak isomorphism; going through the proof there, one sees that it extends immediately to $\mnets_p$. Similarly, \cite[Theorem 3]{memoli2022distance} shows that $\dgm$ defines a Lawvere metric on the space of strong isomorphism classes of \emph{compact} mm-spaces, but the proof extends to all of $\mmspaces$ without change.
\end{remark}

The metric properties for $\dgm$ on the space of measure networks are more subtle.

\begin{proposition}\label{prop:dgm_networks}
On the space $\mnets$, $\dgm$ defines a Lawvere metric such that if $\dgm(\mX,\mY) = 0$ then $\mX$ and $\mY$ are weakly isomorphic and if $\mX$ and $\mY$ are strongly isomorphic, then $\dgm(\mX,\mY) = 0$. However, it is possible for weakly isomorphic $\mX$ and $\mY$ to have $\dgm(\mX,\mY) > 0$. 
\end{proposition}

\begin{proof}
That $\dgm(\mX,\mX) = 0$ is obvious and the proof of the triangle inequality follows as in the mm-space case~\cite{memoli2022distance}, hence $\dgm$ is a Lawvere metric. By Remark \ref{rmk:dgw_dgm_inequality}, $\dgm(\mX,\mY) = 0$ implies $\dgw(\mX,\mY) = 0$, hence $\mX$ and $\mY$ are weakly isomorphic. It is also easy to show that if $\mX$ and $\mY$ are strongly isomorphic, then $\dgm(\mX,\mY) = 0$.

It remains to show that weak isomorphism of $\mX$ and $\mY$ does not necessarily imply $\dgm(\mX,\mY) =0$. This can be done by example: take $X = \{x\}$ with uniform measure and $\omega_X$ identically zero, and $Y = \{y,y'\}$ with uniform measure and $\omega_Y$ identically zero. It follows by a (trivial)  computation that  $\dgw(\mX,\mY) = 0$, hence $\mX$ and $\mY$ are weakly isomorphic, but the set of measure preserving maps $\mathcal{T}(\mu_X,\mu_Y)$ is empty, hence $\dgm(\mX,\mY) = \infty$. 
\end{proof}

\begin{example}\label{ex:nonzero_weak_iso}
With slightly more work, one can construct weakly isomorphic spaces $\mX$ and $\mY$ such that $\dgm(\mX,\mY)$ is \emph{finite} but nonzero. Indeed, take
\[
X = \{x,x',x''\}, \quad \mu_X(x) = \frac{1}{2},\, \mu_X(x') = \mu_X(x'') = \frac{1}{4}, \quad \omega_X = \begin{blockarray}{cccc}
x & x' & x'' \\
\begin{block}{(ccc)c}
  1 & 0 & 0 & x \\
  0 & 0 & 0 & x' \\
  0 & 0 & 0 & x'' \\
\end{block}
\end{blockarray}
\]
and 
\[
Y = \{y,y',y''\}, \quad \mu_Y(y) = \mu_Y(y')= \frac{1}{4},\, \mu_Y(y'') = \frac{1}{2}, \quad \omega_Y = \begin{blockarray}{cccc}
y & y' & y'' \\
\begin{block}{(ccc)c}
  1 & 1 & 0 & y \\
  1 & 1 & 0 & y' \\
  0 & 0 & 0 & y'' \\
\end{block}
\end{blockarray}.
\]
Then one can show that $X$ and $Y$ are weakly isomorphic, but both of the two possible preserving maps $\phi:X \to Y$ have positive distortion.
\end{example}

In Section \ref{sec:gw_as_gm}, we give more precise comparisons between Gromov-Wasserstein and Gromov-Monge distances under additional assumptions on the measure networks. 

\subsection{Related Work}\label{sec:literature}

We defined Gromov-Monge distances between mm-spaces in our previous paper \cite{memoli2022distance}. There, the focus was on restrictions to various subcategories of mm-spaces, and GM distances were mainly used as a framing device to motivate certain inverse problems for mm-spaces. Gromov-Monge distances have since arisen in theoretical and applied contexts, and we survey those appearances here. 

Several articles consider the following problem, which is inspired by the well known work of Brenier in the classical optimal transport setting~\cite{brenier1991polar}. 

\begin{problem}[Monge Map Problem for GW]\label{prob:GM}
Given a class of measure networks, determine whether Gromov-Wasserstein distance between any two elements is realized by a Monge map. That is, for a class $\catC$ of measure networks, we would like to show whether or not it is the case that for all $\mX,\mY \in \catC$, there exists $\phi \in \mathcal{T}(\mu_X,\mu_Y)$ such that $\dgw(\mX,\mY) = \mathrm{dis}_p(\phi)$ (for some $p \in [1,\infty]$).
\end{problem}

Finding a class of mm-spaces with an affirmative answer to Problem \ref{prob:GM} was first posed by Sturm as a ``Challenge" in \cite[Challenge 3.6]{sturm2023space}. Moreover, \cite[Challenge 5.27]{sturm2023space}) asks one to solve the problem specifically for the class of finite mm-spaces of fixed cardinality, with uniform measures. The first solution to Problem \ref{prob:GM} appears as \cite[Theorem 9.21]{sturm2023space} for the class of measure networks $\mathcal{X}$ with $X \subset \R^n$, $\omega_X(x,x') = \|x - x'\|^2$ (squared Euclidean distance) and $\mu_X$ absolutely continuous with respect to Lebesgue measure and rotationally invariant. This result, and all other results described in the remainder of this subsection, are valid for $\mathrm{GW}_2$, specifically.

A solution to the Monge map problem for the class measure networks $\mathcal{X}$ with $X$ a finite subset of the real line of some fixed finite cardinality, $\omega_X(x,x') = |x - x'|^2$ and $\mu_X$ uniform was proposed in \cite{vayer2019sliced}, with a view toward a ``sliced" approximation of GW distance (similar to \define{sliced Wasserstein distances} \cite{bonneel2015sliced}). In particular, it was claimed that the optimal Monge map is always order-preserving or order-reversing. However, a counterexample to this stronger claim was recently demonstrated~\cite{beinert2022assignment}. The solution to Problem \ref{prob:GM} nonetheless holds in this setting, and was shown in the thesis of Vayer to also hold for higher-dimensional finite Euclidean mm-spaces \cite[Theorem 4.1.2]{vayer2020contribution}. That is, for the class of mm-spaces $\mathcal{X}$ with $X$ a subset of a Euclidean space (of arbitrary finite dimension) of fixed finite cardinality, $\omega_X(x,x') = \|x-x'\|^2$ and $\mu_X$ uniform, it was shown that GW distance is always realized by a permutation (i.e., a Monge map). 

In \cite{vayer2020contribution}, Vayer also considered Euclidean measure networks whose measures are assumed to be compactly supported and to have density with respect to Lebesgue measure, and whose network functions are squared Euclidean distance---we refer to this as the \define{quadratic Euclidean class}---and Euclidean measure networks whose network functions are (standard) inner products---we refer to this as the \define{inner product Euclidean class}. Vayer proves the existence of Monge maps under additional assumptions (the existence of an optimal coupling with certain properties) in both of these classes \cite[Theorem 4.2.3 and Proposition 4.2.4]{vayer2020contribution}. Recently, general detailed solutions to Problem \ref{prob:GM} were obtained in the quadratic and inner product Euclidean classes by Dumont, Lacombe, and Vialard in \cite[Theorems 3.2 and 3.6]{dumont2024existence}, as consequences of technical theorems on existence of Monge maps for optimal transport costs defined by submersions. Finally, in \cite{salmona2022gromov}, Salmona, Delon and Desolneux study GW distances between measure networks in the quadratic and inner product Euclidean classes whose measures are Gaussians. An explicit formula is derived in the inner product class \cite[Proposition 4.1]{salmona2022gromov} and it is shown that if one restricts the feasible set of couplings to those couplings which are themselves Gaussian, the same solution is optimal in the quadratic class \cite[Theorem 4.1]{salmona2022gromov}.

Gromov-Monge distances, or closely related variants, have recently appeared in the machine learning literature. In \cite{hur2024reversible}, the GM distance is symmetrized by defining a new distance, called \define{Reversible Gromov-Monge (RGM) distance}, which involves optimization over a pair of measure-preserving maps satisfying a certain consistency condition. In particular, one always has $\mathrm{GM}_2(\mX,\mY) \leq \mathrm{RGM}(\mX,\mY)$. Several theoretical properties of RGM distance are derived, and applications to simulation-based inference are described, where the idea is to use optimal measure-preserving maps to design transform samplers. Similarly, a two-map variant of GM distance is introduced in \cite{zhang2022cycle}, where the main distinction from RGM is the inclusion of a term which penalizes the maps according to how far they are from being inverses to one another. A benefit of GM-type distances, which is taken advantage of in both \cite{hur2024reversible} and \cite{zhang2022cycle}, is that the spaces of admissible mappings can be parameterized as neural networks. Learning optimal measure-preserving maps can then take advantage of efficient neural network training algorithms. 

Finally, we note that isometry-invariant versions of Monge distance have appeared in the literature previously in the context of registering images \cite{haker2004optimal} and anatomical surfaces \cite{boyer2011algorithms,al2013continuous}. We show in Section \ref{sec:GM_for_Euclidean_spaces} that these metrics are bi-H\"{o}lder equivalent to an alternative version of GM distance, as defined in Section \ref{sec:em_GM}. More details on these isometry-invariant metrics and their connections to GM distances are provided below in Remark \ref{rmk:more_examples}.

\section{Comparisons Between Gromov-Monge and Gromov-Wasserstein Distances}\label{sec:gw_as_gm}

A natural relaxation of the Monge Map Problem \ref{prob:GM} is: 

\begin{problem}[Equality of GW and GM]\label{prob:gw_equals_gm}
Given a class of measure networks $\catC$, determine whether $\dgw(\mX,\mY) = \dgm(\mX,\mY)$ for all $\mX,\mY \in \catC$. 
\end{problem}

This is indeed a relaxation of Problem \ref{prob:GM}; if there exists a measure-preserving map $\phi:X \to Y$ such that $\dgw(\mX,\mY) = \mathrm{dis}_p(\phi)$, then $\dgw(\mX,\mY) \geq \dgm(\mX,\mY)$ and Remark \ref{rmk:dgw_dgm_inequality} implies that this is actually an equality. This section treats Problem \ref{prob:gw_equals_gm} and other related problems.

\subsection{Non-Atomic Spaces}

We first address the question of equality of Gromov-Wasserstein and Gromov-Monge distances in the setting of non-atomic spaces---recall that a measure is \define{non-atomic} if it assigns zero mass to any singleton.

\begin{theorem}\label{thm:nonatomic_equivalence}
Let $\mX$ and $\mY$ be measure networks in $\mnets_p$ such that $\mu_X$ and $\mu_Y$ are non-atomic. Then, for $p \in [1,\infty)$,
$
\dgm(\mX,\mY) = \dgw(\mX,\mY). 
$
\end{theorem}

This result is analogous to a theorem of Pratelli \cite{pratelli2007equality}, who showed equality of Wasserstein and Monge optimal transport for nonatomic measures when the underlying cost is assumed to be continuous, generalizing a result of Ambrosio \cite[Theorem 2.1]{ambrosio2003lecture} (remarkably, Pratelli's theorem allows for unbounded costs, which can even take the value $\infty$). Our theorem generalizes recent work of Hur, Guo and Liang \cite{hur2021reversible}, which shows equality of $\mathrm{GW}_2$ and $\mathrm{GM}_2$ for non-atomic measure networks with continuous and bounded network functions whose underlying spaces are subsets of $\R^d$ (this follows by combining Theorem 3 and Proposition 1 of \cite{hur2021reversible}). Notably, we make no continuity or boundedness assumption on network functions in Theorem \ref{thm:nonatomic_equivalence}. In~\cite[Theorem 5.5]{hur2024reversible}, from the published version of~\cite{hur2021reversible}, the authors give a refinement of Theorem~\ref{thm:nonatomic_equivalence}, at a similar level of generality, citing a preprint version of the present paper as an inspiration for the proof of their result.

Let $I = [0,1]$ and let $\lambda$ denote Lebesgue measure, restricted to $I$. We will consider the product measure space $(I^2,\lambda^2) = (I \times I, \lambda \otimes \lambda)$. The strategy of the proof is to transform the general problem into the simpler subproblem on the class of measure networks of the form $(I^2,\omega,\lambda^2)$, via a classical isomorphism theorem. Such a strategy was proposed by Gangbo in \cite[Proposition A.3]{gangbo1999monge} to prove equality of Wasserstein and Monge distances in the classical optimal transport setting. However, Pratelli observes in \cite[Section 1.2]{pratelli2007equality} that such an approach cannot work in general, due to the fact that the maps appearing in the isomorphism theorem may be discontinuous. We will show that this lack of continuity is no longer a problem in the GW/GM setting, leading to a relatively simpler proof than the classical optimal transport results \cite{ambrosio2003lecture,pratelli2007equality}. We will require a few lemmas.

\begin{lemma}\label{lem:model_spaces}
    Let $\mX$ be a measure network with non-atomic probability measure $\mu_X$. Then there exists a measure network 
 $\mathcal{I}_\mathcal{X} = (I^2,\omega_{I^2}^X,\lambda^2)$ which is strongly isomorphic to $\mX$.
\end{lemma}

\begin{proof}
It is a standard fact of measure theory that, since $(X,\mu_X)$ is a Polish space with non-atomic probability measure, there exists a measure-preserving bijection $\phi: I \to X$ (with respect to $\lambda$ and $\mu_X$) such that its inverse is also measure-preserving (i.e., $\phi$ is a \emph{measure space isomorphism}) \cite[Ch. 15, Theorem 16]{royden1968real}. Likewise, we have a measure space isomorphism $\psi:I \to I^2$, with respect to $\lambda$ and $\lambda^2$. Then $\tau \coloneqq \phi \circ \psi^{-1}:I^2 \to X$ is a measure space isomorphism with respect to $\lambda^2$ and $\mu_X$. We then define 
\[
\omega_{I^2}^X(s,t) \coloneqq \omega_X(\tau(s),\tau(t))
\]
for $s,t \in I$. Then $\tau$ is a strong isomorphism, by definition. 
\end{proof}

\begin{remark}\label{rmk:square_vs_interval}
Following the proof of Lemma \ref{lem:model_spaces}, one sees that we could have alternatively identified $\mX$ strongly isomorphically with a measure network of the form $(I,\omega_{I}^X,\lambda)$---indeed, this is a common method for \emph{parameterizing} m-nets or mm-spaces (see \cite[Section 2.5.1]{chowdhury2019gromovNetworks} or \cite[Lemma 5.3]{sturm2023space}). The reason for modeling our spaces over the square $(I^2, \lambda^2)$ is so that we can apply a technical lemma (Lemma \ref{lem:map_approximation}), as is further explained in Remark \ref{rmk:dimension_dependence}.
\end{remark}

The next result then follows by Theorem \ref{thm:summary} and Proposition \ref{prop:dgm_networks} (in particular, the triangle inequalities and the identification of strongly isomorphic spaces under both $\dgw$ and $\dgm$).

\begin{lemma}\label{lem:equality_for_models}
Let $\mX$ and $\mY$ be Polish measure networks with non-atomic probability measures, and let $\mIX$ and $\mIY$ be as in Lemma \ref{lem:model_spaces}. Then $\dgw(\mX,\mY) = \dgw(\mIX,\mIY)$ and $\dgm(\mX,\mY) = \dgm(\mIX,\mIY)$.
\end{lemma}

The following is a special case of \cite[Theorem 1.1 (i)]{brenier2003p}. 

\begin{lemma}[\cite{brenier2003p}]\label{lem:map_approximation}
For any $\pi \in \mathcal{C}(\lambda^2,\lambda^2)$, there exists a sequence of measure-perserving maps $\phi_m:I^2 \to I^2$ such that $(\mathrm{id}_{I^2} \times \phi_m)_\# \lambda^2$ converges weakly to $\pi$. 
\end{lemma}

\begin{remark}\label{rmk:dimension_dependence}
In the original work of Brenier and Gangbo~\cite{brenier2003p}, the result was stated for approximation of couplings over the cube $I^d$  for $d \geq 2$, as this is the general setting for other results in that paper. This is the reason that we have chosen to identify an arbitrary m-net $\mX$ with an m-net $\mathcal{I}_\mX$ over the square $(I^2,\lambda^2)$ in Lemma \ref{lem:model_spaces}, rather than over the interval $(I,\lambda)$ (see Remark \ref{rmk:square_vs_interval}). In fact, checking the details of the proof of \cite[Theorem 1.1]{brenier2003p} one sees that it applies also to the $d=1$ case (i.e., to $(I,\lambda)$). However, in order to give a precise reference, we have opted to work with the $d=2$ case. 
\end{remark}

\begin{proof}[Proof of Theorem \ref{thm:nonatomic_equivalence}]
It suffices to consider the case of measure networks $\mathcal{I} = (I^2,\omega,\lambda^2)$ and $\mathcal{I}' = (I^2,\omega',\lambda^2)$  $\mnets_p$. Indeed, if the result holds for measure networks over $(I^2,\lambda^2)$, then for general non-atomic measure networks $\mX$ and $\mY$, one has
\[
\dgw(\mX,\mY) = \dgw(\mIX,\mIY) = \dgm(\mIX,\mIY) = \dgm(\mX,\mY),
\]
by Lemma \ref{lem:equality_for_models}.

Let $\pi$ be an optimal coupling for $\mathcal{I}$ and $\mathcal{I}'$ (see Remark \ref{rmk:optimal_couplings}). By Lemma \ref{lem:map_approximation}, there is a sequence of measure-preserving maps $\phi_m:I^2 \to I^2$ such that $\pi_m\coloneqq(\mathrm{id}_{I^2} \times \phi_m)_\# \xrightarrow[]{m \to \infty} \pi$, in the weak topology. It is proved in~\cite[Lemma 2.3]{chowdhury2019gromovNetworks} that the distortion function $\mathrm{dis}_p$ is continuous in the weak topology, when it is defined with respect to measure networks in $\mnets_\infty$. The same proof goes through for measure networks in $\mnets_p$, so we have
 \[
    \dgm(\mathcal{I},\mathcal{I}') \leq \lim_{m \to \infty} \mathrm{dis}_p(\pi_m) = \mathrm{dis}_p(\pi) = \dgw(\mathcal{I},\mathcal{I}').
 \]
 By the general bound $\dgw(\mathcal{I},\mathcal{I}') \leq \dgm(\mathcal{I},\mathcal{I}')$ (Remark \ref{rmk:dgw_dgm_inequality}), this shows $\dgw(\mathcal{I},\mathcal{I}') = \dgm(\mathcal{I},\mathcal{I}')$ and completes the proof of the theorem.
\end{proof}

\subsection{Discrete Spaces}

We now consider the scenario which is the ``opposite" of the non-atomic setting: finite spaces with discrete measures. For this subsection, let $\catC_{\mathrm{fin}}$ denote the class of measure networks $\mX$ finite and $\mu_X$ uniform; we put no additional restriction on network functions at this point. We begin with a simple structural result, characterizing infinite values and asymmetry for Gromov-Monge distances.

\begin{proposition}
Let $\mX,\mY \in \catC_{\mathrm{fin}}$ with $|X| = m$ and $|Y| = n$. Then $\dgm(\mX,\mY) < \infty$ if and only if $n$ divides $m$. In the case $m = n$, $\dgm(\mX,\mY) = \dgm(\mY,\mX) < \infty$. 
\end{proposition}

\begin{proof}
Suppose that $n$ divides $m$, say $m = kn$. Then any function $\phi:X \to Y$ whose fibers all have cardinality $k$ is a measure-preserving map; in particular, $\mathcal{T}(\mu_X,\mu_Y) \neq \emptyset$, so $\dgm(\mX,\mY) < \infty$. Conversely, let $\phi:X \to Y$ be a measure-preserving map. For any $y \in Y$, we have 
\[
\frac{1}{n} = \mu_Y(y) = \mu_X(\phi^{-1}(y)) = |\phi^{-1}(y)| \frac{1}{m},
\]
hence $m = |\phi^{-1}(y)| n$. 

If $m = n$, then the measure preserving maps are exactly the bijections $\phi:X \to Y$, and 
\[
\mathrm{dis}_{p,\mX,\mY}(\phi) = \mathrm{dis}_{p,\mY,\mX}(\phi^{-1}) < \infty.
\]
\end{proof}

In light of the previous result, it is natural to restrict our attention to the subclass $\catC_n$ consisting of $\mX \in \catC_{\mathrm{fin}}$ of fixed cardinality, $|X| = n$. As was stated in Section \ref{sec:literature}, Problem \ref{prob:GM} has been solved by Vayer in the subclass of $\catC_n$ consisting of measure networks $\mX$ with $X \subset \R^d$ (for some arbitrary dimension $d$) and $\omega_X(x,x') = \|x - x'\|^2$: in this subclass, GW 2-distance is realized by a Monge map (i.e., a permutation) \cite[Theorem 4.1.2]{vayer2020contribution}. The proof uses ideas of \cite{maron2018probably} and, in particular, relies on the observation that squared Euclidean distance matrices are conditionally negative definite. We now prove a similar result for the subclass of $\catC_n$ consisting of m-nets with symmetric positive definite network functions. That is, for $\mX \in \catC_n$, we pick an ordering $(x_1,\ldots,x_n)$ of $X$ and  consider the network function as a matrix $\omega_X \in \R^{n \times n}$ by setting $\omega_X(i,j) = \omega_X(x_i,x_j)$ (see Example \ref{ex:nonzero_weak_iso}). We say that $\omega_X$ if \define{symmetric positive definite} if it is symmetric positive definite as a matrix.

\begin{proposition}\label{prop:positive_definite}
For any $n \in \mathbb{Z}_{>0}$, consider $\mX,\mY \in \catC_n$ such that their network functions are symmetric positive definite. Then $\mathrm{GW}_2(\mX,\mY)$ is realized by a measure-preserving map. 
\end{proposition}

As an example of a natural symmetric positive definite network function, consider a graph with vertex set $X$, $\mu_X$ uniform and $\omega_X$ given by the heat kernel $\exp(-tL)$ for some parameter $t > 0$, where $L$ is the graph Laplacian matrix. The idea of the proof  comes from \cite[Theorem 2]{chowdhury2021generalized}, which specifically considers graph heat kernels, and \cite[Lemma 4.3]{alvarez2019towards}, which offers a similar computation in the setting of cosine similarity matrices; neither of these results make the connection to Monge maps, although the former derives a bound on sparsity of optimal couplings.

\begin{proof}
As above, consider $\omega_X$ and $\omega_Y$ as matrices in $\R^{n \times n}$. Likewise, we consider probability measures as column vectors $\mu_X,\mu_Y \in \R^n$ (specifically, they are both equal to the column vector with all entries equal to $\frac{1}{n}$). Since it is symmetric positive definite, $\omega_X$ admits a Cholesky decomposition $\omega_X = U_X^T U_X$. Let $\omega_Y = V_Y^T V_Y$ be defined similarly. 

A coupling of $\mu_X$ and $\mu_Y$ can be considered as a matrix $\pi \in \R^{n \times n}$ whose row and column sums agree with $\mu_X$ and $\mu_Y$, respectively. We can express the distortion of $\pi$ as
\begin{align*}
\mathrm{dis}_2(\pi)^2 &= \sum_{i,j,k,\ell} (\omega_X(i,k) - \omega_Y(j,\ell))^2 \pi(i,j) \pi(k,\ell) \\
&= \sum_{i,j,k,\ell} \omega_X(i,k)^2 \pi(i,j) \pi(k,\ell) + \sum_{i,j,k,\ell} \omega_Y(j,\ell)^2 \pi(i,j) \pi(k,\ell) \\
&\hspace{2.7in} - 2 \sum_{i,j,k,\ell} \omega_X(i,k)\omega_Y(j,\ell) \pi(i,j) \pi(k,\ell) \\
&= \sum_{i,k} \omega_X(i,k)^2 \mu_X(i)\mu_X(k) + \sum_{j,\ell} \omega_Y(j,\ell)^2 \mu_Y(j)\mu_Y(\ell) \\
&\hspace{2.7in} - 2 \sum_{i,j,k,\ell} \omega_X(i,k)\omega_Y(j,\ell) \pi(i,j) \pi(k,\ell).
\end{align*}
Since the first two terms do not depend on $\pi$, minimizing $\mathrm{dis}_2(\pi)^2$ is equivalent to maximizing the quanity $\sum_{i,j,k,\ell} \omega_X(i,k)\omega_Y(j,\ell) \pi(i,j) \pi(k,\ell)$ over all couplings $\pi$. Let $\left<\cdot,\cdot\right>$ denote the Frobenius inner product on $\R^{n \times n}$ and $\|\cdot\|$ the associated norm. Then our object is to maximize 
\begin{align*}
    \sum_{i,j,k,\ell} \omega_X(i,k)\omega_Y(j,\ell) \pi(i,j) \pi(k,\ell) &= \left<\omega_X \pi , \pi \omega_Y\right> = \left<U_X^T U_X \pi, \pi V_Y^T V_Y \right> \\
    &= \left<U_X \pi V_Y^T, U_X \pi V_Y^T \right> = \|U_X \pi V_Y^T\|^2.
\end{align*}
The function $\pi \mapsto \|U_X \pi V_X^T\|^2$ is a convex function on $\R^{n \times n}$. Moreover, considered as a space of matrices, the set of couplings $\mathcal{C}(\mu_X,\mu_Y)$ is given by
\[
\mathcal{C}(\mu_X,\mu_Y) = \left\{\pi \in \R^{n \times n} \mid \sum_{i} \pi(i,j) = \sum_j \pi(i,j) = \frac{1}{n} \mbox{ and } \pi(i,j) \geq 0 \; \forall \; i,j \right\};
\]
that is, it is a rescaling of the set of bistochatic matrices by a factor of $\frac{1}{n}$. By Birkhoff's theorem, $\mathcal{C}(\mu_X,\mu_Y)$ is the convex hull of the set of scaled (by $\frac{1}{n}$) permutation matrices. These correspond to couplings which are induced by measure-preserving mappings $X \to Y$. By standard optimization theory, the convex function $\pi \mapsto \|U_X \pi V_X^T\|$ is always maximized at an extremal point of its constraint polytope \cite{Benson2001}. This proves the claim.
\end{proof}

\subsection{The General Metric Measure Space Setting}

We consider a variant of Gromov-Monge distance on the class of mm-spaces which allows for a certain notion of mass splitting. 

\begin{definition}
Let $(X,d_X)$ be a metric space, let $Z$ be a set and let $\phi:Z \rightarrow X$ be a function. The associated \define{pullback pseudometric} is the function $\phi^\ast d_X: Z \times Z \rightarrow \R$ defined by
\begin{linenomath}\begin{equation*}
(\phi^\ast d_X)(z,z') = d_X(\phi(z),\phi(z')).
\end{equation*}\end{linenomath}
It is easy to check that $\phi^\ast d_X$ indeed satisfies the axioms of a pseudometric.
\end{definition}

\begin{definition}
Let $\mX$ be an mm-space. A \define{mass splitting} of $\mX$ is a measure network $\mZ$ such that there exists a measure-preserving map $\rho:Z \rightarrow X$ with the property that $\omega_Z = \rho^\ast d_X$.
\end{definition}

\begin{theorem}\label{thm:monge_wasserstein_comparison}
Let $\mX$ and $\mY$ be metric measure spaces. Then
\begin{linenomath}\begin{equation*}
\dgw(\mX,\mY) = \inf_{\mZ} \dgm(\mZ,\mY),
\end{equation*}\end{linenomath}
where the infimum is taken over mass-splittings of $\mX$.
\end{theorem}

Similar ideas of splitting mass to realize Gromov-Wasserstein distance through Monge maps were employed in \cite{chowdhury2019gromov} to construct explicit geodesics between networks in a shape analysis setting. The algorithm used in \cite{chowdhury2019gromov} is an  implementation of the geodesic characterization first obtained in \cite[Theorem 3.1]{sturm2023space}. The result is also reminiscent of \cite[Proposition 2.7.5]{chowdhury2022distances}, which considers a similar characterization of a Gromov-Hausdorff-like distance on the space of networks (without measures). 

\begin{example}
To illustrate the idea of the theorem with a simple example, consider the space $\mX$ consisting of a single point and the space $\mY$ consisting of two points $Y=\{y_1,y_2\}$ with $d_Y(y_1,y_2) =1$ and with uniform weights. The Gromov-Wasserstein $p$-distance between the spaces is realized by the unique coupling $\pi$ satisfying $\pi((x,y_j)) = \frac{1}{2}$ so that $\dgw(\mX,\mY) = \frac{1}{2^{1/p}}$, while $\dgm(\mX,\mY) = \infty$ since $\mT(\mu_X,\mu_Y) = \emptyset$. On the other hand, consider the mass splitting $\mZ$ with $Z = \{z_1,z_2\}$, $\omega_Z(z_1,z_2) = 0$ and with uniform weights. It is easy to check that $\mZ$ is a mass-splitting of $\mX$ with $\rho(z_j) = x$. Moreover, $\dgm(\mZ,\mY) = \frac{1}{2^{1/p}}$ and this quantity is realized by the measure-preserving map $\phi(z_j) = y_j$. This general setup is illustrated by the following diagram:
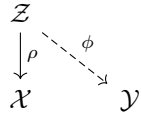
\begin{center}
\begin{tikzcd}
\mZ  \arrow[d, "\rho"] \arrow[rd, "\phi", dashrightarrow] & \\
\mX & \mY
\end{tikzcd}
\end{center}
\end{example}

\begin{proof}[Proof of Theorem \ref{thm:monge_wasserstein_comparison}]
Given a measure coupling $\pi$ of $\mu_X$ and $\mu_Y$, define a mass splitting $\mZ$ by setting $Z=\mathrm{supp}(\pi) \subset X \times Y$, $\mu_Z = \pi$, $\rho = \rho_X|_Z$ (with $\rho_X: X \times Y \rightarrow X$ denoting projection onto the first coordinate) and $\omega_Z = \rho^\ast d_X$. Then $\phi=\rho_Y|_Z:Z \rightarrow Y$ (with $\rho_Y:X \times Y \to Y$ the projection map) is a measure-preserving map with
\begin{linenomath}\begin{align*}
\mathrm{dis}_{p,\mZ,\mY}(\phi)^p &= \iint_{Z \times Z} |\omega_Z(z,z') - d_Y(\phi(z),\phi(z'))|^p \mu_Z(dz)\mu_Z (dz')\\
&= \iint_{(X \times Y)^2} |d_X(x,x') - d_Y(y,y')|^p \pi(dx \times dy) \pi(dx' \times dy') = \mathrm{dis}_{p,\mX,\mY}(\pi)^p.
\end{align*}\end{linenomath}
We conclude that
\begin{linenomath}\begin{equation*}
\dgw(\mX,\mY) \geq \inf_{\mZ} \dgm(\mZ,\mY).
\end{equation*}\end{linenomath}

Conversely, let $\mZ = (Z,\omega_Z,\mu_Z)$ be a mass splitting of $\mX$, with measure-preserving map $\rho:Z \to X$ such that $\omega_Z = \rho^\ast d_X$, and let $\phi:Z \rightarrow Y$ be a measure-preserving map (we assume that one exists; otherwise, the desired inequality $\dgw(\mX,\mY) \leq \dgm(\mZ,\mY)$ follows trivially). We define a probability measure $\pi$ on $X \times Y$ as
$\pi = (\rho \times \phi)_\# \mu_Z$. Then $\pi$ defines a measure coupling of $\mu_X$ and $\mu_Y$. Indeed,
\[
(\rho_X)_\# \pi = (\rho_X)_\# (\rho \times \phi)_\# \mu_Z = (\rho_X \circ (\rho \times \phi))_\# \mu_Z =  \rho_\# \mu_Z = \mu_X,
\]
and, by a similar argument, we also have $(\rho_Y)_\# \pi = \mu_Y$. Consider the following calculation (we suppress all function arguments to condense notation): 
\begin{linenomath}\begin{align}
\iint_{(X \times Y)^2} |d_X - d_Y|^p \pi \otimes \pi &= \iint_{(X \times Y)^2} |d_X - d_Y|^p (\rho \times \phi)_\# \mu_Z \otimes (\rho \times \phi)_\# \mu_Z \nonumber \\
&=\iint_{(X \times Y)^2} |d_X - d_Y|^p \left((\rho \times \phi) \times (\rho \times \phi)\right)_\# \mu_Z \otimes \mu_Z \label{eqn:monge_wasserstein_2} \\
&= \iint_{Z \times Z} |d_X \circ (\rho \times \rho) - d_Y \circ (\phi \times \phi)|^p \mu_Z \otimes \mu_Z \label{eqn:monge_wasserstein_3} \\
&= \iint_{Z \times Z} |\omega_Z - d_Y \circ (\phi \times \phi)|^p \mu_Z \otimes \mu_Z \label{eqn:monge_wasserstein_4}.
\end{align}\end{linenomath}
Equality \eqref{eqn:monge_wasserstein_2} follows from the fact that $(\pi \times \phi)_\# \mu_Z \otimes (\pi \times \phi)_\# \mu_Z$ and $((\pi \times \phi) \times (\pi \times \phi))_\# \mu_Z \otimes \mu_Z$ define equivalent measures on $(X \times Y)^2$ and \eqref{eqn:monge_wasserstein_3} and \eqref{eqn:monge_wasserstein_4} both follow from the change-of-variables formula, with the latter also using $d_X \circ (\rho \times \rho) = \omega_Z$. This calculation shows that
\begin{linenomath}\begin{equation*}
\dgw(\mX,\mY) \leq \inf_{\mZ} \dgm(\mZ,\mY)
\end{equation*}\end{linenomath}
and the proof is therefore complete.
\end{proof}

\section{An Embedding Formulation of Gromov-Monge Distance}\label{sec:gromovization}

In this section, we restrict our attention to mm-spaces (rather than general m-nets) and study alternative notions of distance between them.

\subsection{Gromov-Monge Distances from Joint Embeddings}\label{sec:em_GM}

In \cite{sturm2006geometry}, Sturm defined an alternative metric for comparing mm-spaces, which is analogous to the embedding formulation of Gromov-Hausdorff distance. For mm-spaces $\mX$ and $\mY$, let 
\begin{linenomath}\begin{equation*}
\dst(\mX,\mY) \coloneqq \inf_{Z,\phi_X,\phi_Y} \dw^Z((\phi_X)_\# \mu_X, (\phi_Y)_\# \mu_Y),
\end{equation*}\end{linenomath}
where the infimum is taken over all isometric embeddings $\phi_X:X \rightarrow Z$ and $\phi_Y:Y \rightarrow Z$ into some metric space $Z$ and $\dw^Z$ is the usual Wasserstein $p$-distance between probability measures on $Z$. We analogously define a new extended quasi-metric $\dstm$ on the space of mm-spaces by
\begin{linenomath}\begin{equation*}
\dstm(\mX,\mY)\coloneqq\inf_{Z,\phi_X,\phi_Y} \dm^Z((\phi_X)_\# \mu_X, (\phi_Y)_\# \mu_Y),
\end{equation*}\end{linenomath}
where the infimum is once again taken over isometric embeddings into a common ambient metric space. 

\begin{theorem}\label{thm:gm_equals_gw_embedding}
Let $\mX$ and $\mY$ be mm-spaces with non-atomic probability measures. Then for $p \in [1,\infty)$, $\dst(\mX,\mY) = \dstm(\mX,\mY)$. 
\end{theorem}

\begin{proof}
Clearly, $\dst \leq \dstm$ holds in general; in particular, if $\dst(\mX,\mY) = \infty$ (which is possible without any boundedness assumption on the mm-spaces) then $\dstm(\mX,\mY) = \infty$ as well.  Assume, then, that $\dst(\mX,\mY)$ is finite, and let $\phi_X:X \to Z$ and $\phi_Y:Y \to Z$ be isometric embeddings into some metric space $(Z,d_Z)$. It follows from \cite[Lemma 3.3 (ii)]{sturm2006geometry} that, in the context of computing $\dst$, we may assume without loss of generality that $(Z,d_Z)$ is complete and separable---that is, we can restrict the definition to the infimum of embeddings into such spaces without changing the value of the metric. The same proof applies to show that this can be assumed without loss of generality in the context of computing $\dstm$. Since the pushforwards $(\phi_X)_\# \mu_X$ and $(\phi_Y)_\# \mu_Y$ are nonatomic (this follows by injectivity and the assumption that the original measures were nonatomic) and the function $Z \times Z \to \R$ defined by $(z,z') \mapsto d_Z(z,z')^p$ is continuous, Pratelli's result \cite[Theorem B]{pratelli2007equality} implies that $\dw^Z((\phi_X)_\# \mu_X, (\phi_Y)_\# \mu_Y) = \dm^Z((\phi_X)_\# \mu_X, (\phi_Y)_\# \mu_Y)$. Because this equality holds for every choice of $Z,\phi_X,\phi_Y$, passing to the infimum over such embeddings yields $\dst(\mX,\mY) = \dstm(\mX,\mY)$.
\end{proof}

\subsection{Comparing Gromov-Monge Distances}

Here, we show that the two notions of Gromov-Monge distance considered in this paper are comparable for arbitrary mm-spaces, but are not equivalent in general. The following is a Gromov-Monge version of \cite[Theorem 5.1(g)]{memoli2011gromov}.

\begin{theorem}\label{thm:gromov_monge_reformulation}
For mm-spaces $\mX$ and $\mY$ and for any $p \in [1,\infty]$,
\begin{equation}\label{eqn:gromov_monge_reformulation_1}
\frac{1}{2}\dgm(\mX,\mY) \leq \dstm(\mX,\mY).
\end{equation}
In the case that $p=\infty$, we have
\begin{equation}\label{eqn:gromov_monge_reformulation_2}
\frac{1}{2}\mathrm{GM}_{\infty}(\mX,\mY) = \mathrm{GM}^{\mathrm{em}}_{\infty}(\mX,\mY).
\end{equation}
\end{theorem}

\begin{proof}
If $\dstm(\mX,\mY) = \infty$ then \eqref{eqn:gromov_monge_reformulation_1} is trivially satisfied, so let us assume that $\dstm(\mX,\mY) < \infty$. To prove that  \eqref{eqn:gromov_monge_reformulation_1} holds under this assumption, we will show that whenever $\dstm(\mX,\mY) < r$, there exists $\phi \in \mathcal{T}(\mu_X,\mu_Y)$ such that $\|d_X - d_Y\|_{L^p(\pi_\phi\otimes \pi_\phi)} \leq 2r$. If $\dstm(\mX,\mY) < r$ then we can find isometries $\phi_X:X \rightarrow Z$ and $\phi_Y:Y \rightarrow Z$ into a metric space $Z$ such that $\dm^Z((\phi_X)_\# \mu_X, (\phi_Y)_\# \mu_Y) < r$. We may as well assume that $X,Y \subset Z$ and that $\mu_X$ and $\mu_Y$ are measures on $Z$ with $\supp(\mu_X)=X$ and $\supp(\mu_Y)=Y$. By definition of $\dm^Z$, there exists $\phi \in \mathcal{T}(\mu_X,\mu_Y)$ such that
$
\|d_Z\|_{L^p(\pi_\phi)} < r.
$
Now note that for all $x,x' \in X$, the triangle inequality in $Z$ implies that
\begin{linenomath}\begin{equation*}
|d_Z(x,x') - d_Z(\phi(x),\phi(x'))| \leq d_Z(x,\phi(x)) + d_Z(x',\phi(x')).
\end{equation*}\end{linenomath}
Putting this together with the triangle inequality for the $L^p$ norm, we have 
\begin{linenomath}\begin{equation*}
\|d_X - d_Y\|_{L^p(\pi_\phi \otimes \pi_\phi)} \leq 2\|d_Z\|_{L^p(\pi_\phi)}<2r.
\end{equation*}\end{linenomath}
This establishes (\ref{eqn:gromov_monge_reformulation_1}).

Now we wish to show that $\frac{1}{2} \mathrm{GM}_\infty(\mX,\mY) \geq \mathrm{GM}_\infty^{\mathrm{em}}(\mX,\mY)$. If $\mathrm{GM}_\infty(\mX,\mY) = \infty$ then we are done, so assume not. Let $\phi_0:X \rightarrow Y$ be any measure preserving map with 
\begin{linenomath}\begin{equation*}
\|d_X - d_Y\|_{L^\infty(\mathrm{supp}(\pi_{\phi_0} \otimes \pi_{\phi_0}))} = \sup_{x,x' \in X} |d_X(x,x') - d_Y(\phi_0(x),\phi_0(x'))|=2r.
\end{equation*}\end{linenomath}
The claim follows if we are able to construct a metric space $(Z,d_Z)$ and isometric embeddings $\phi_X$ and $\phi_Y$ such that
$
\mathrm{M}_\infty^Z((\phi_X)_\# \mu_X,(\phi_Y)_\# \mu_Y) \leq r.
$
Let $Z$ denote the disjoint union of $X$ and $Y$ and define a metric $d_Z$ on $Z$ by setting $d_Z|_{X \times X}=d_X$, $d_Z|_{Y \times Y} = d_Y$ and
\begin{linenomath}\begin{align*}
d_Z(x,y)=d_Z(y,x)=\inf_{x' \in X} \left\{d_X(x,x') + r + d_Y(\phi_0(x'),y)\right\}
\end{align*}\end{linenomath}
for any $x \in X$ and $y \in Y$. Then
\begin{linenomath}\begin{align*}
\mathrm{M}_\infty^Z(\mu_X,\mu_Y) &=\inf_{\phi \in \mathcal{T}(\mu_X,\mu_Y)} \sup_{x \in X} d_Z(x,\phi(x)) \leq \sup_{x \in X} d_Z(x,\phi_0(x)) \\
&= \sup_{x \in X} \inf_{x' \in X}\left\{d_X(x,x')+r + d_Y(\phi_0(x'),\phi_0(x))\right\} = r.
\end{align*}\end{linenomath}
This completes the proof.
\end{proof}

\begin{example}[$\dgm$ and $\dgm^{\mathrm{em}}$ are Not BiLipschitz Equivalent] \label{examp:no_lipschitz_bound}
Consider the  family of mm-spaces $\Delta_n$. Each $\Delta_n$ consists of the space $X_n=\{1,\ldots,n\}$ with metric $d_n(i,j)=1-\delta_{ij}$ and measure $\nu_n$ defined by $\nu_n(i)=1/n$. For $p < \infty$,
\begin{linenomath}\begin{equation*}
\dgm^{\mathrm{em}}(\Delta_{2n},\Delta_n) \geq \dst(\Delta_{2n},\Delta_n) \geq \frac{1}{4},
\end{equation*}\end{linenomath}
where the former bound holds generally and the latter is shown in \cite[Claim 5.3]{memoli2011gromov}. On the other hand, for $p < \infty$, $\dgm(\Delta_{2n},\Delta_n) = 1/(2n)^{1/p}$. 
This follows from the fact that $\mathcal{T}(\nu_{2n},\nu_n)$ is simply the set of 2-to-1 maps from $\{1,\ldots,2n\}$ to $\{1,\ldots,n\}$ and all  such maps have the same Monge cost, so that the quantity is obtained by a direct calculation. It follows that $\dgm$ and $\dstm$ are not bi-Lipschitz equivalent for any $p \in [1,\infty)$.
\end{example}

\begin{example}[Tightness of the Inequality]\label{exmp:tightness}
We now show, by example, that the factor of $\frac{1}{2}$ in \eqref{eqn:gromov_monge_reformulation_1} cannot be improved. The analysis provided here is similar to that of~\cite[Remark 5.14]{memoli2011gromov}.

Using the notation of Example \ref{examp:no_lipschitz_bound}, consider the spaces $\Delta_n$ and $\Delta_1 \approx \ast$ (the latter is the one point mm-space). The distance $\dgm(\Delta_n,\ast)$ is the $p$-size of $\Delta_n$ (see Remark \ref{rem:p-size}), which is given explicitly by
\[
\dgm(\Delta_n,\ast) = \left(\sum_{i,j=1}^n d_n(i,j)^p \cdot \frac{1}{n^2} \right)^{1/p} = \left(\sum_{i \neq j} 1^p \cdot \frac{1}{n^2} \right)^{1/p} = \left(n(n-1) \cdot \frac{1}{n^2} \right)^{1/p} = \left(1 - \frac{1}{n}\right)^{1/p}
\]
(this agrees with the more general formula given in \cite[Remark 5.17]{memoli2011gromov}). On the other hand, we claim that $\dstm(\Delta_n,\ast) = \frac{1}{2}$ for all $n$. We then have that, for any $c < 2$, one can take $n$ sufficently large so that $\dgm(\Delta_n,\ast) > c \cdot \dstm(\Delta_n,\ast)$, which shows that the inequality \eqref{eqn:gromov_monge_reformulation_1} is tight.

It remains to derive the value of $\dstm(\Delta_n,\ast)$. In this setting, any isometric embeddings of $\Delta_n$ and $\ast$ into a common metric space amount to choosing distances $\alpha_i$ between $i \in \Delta_n$ and the single point $x$. Given such a choice $\vec{\alpha} = (\alpha_1,\ldots,\alpha_n)$, the Monge distance between the pushforward measures is given by 
\[
\left(\sum_{i=1}^n \alpha_i^p \cdot \frac{1}{n}\right)^{1/p} = \frac{1}{n^{1/p}} \|\vec{\alpha}\|_p,
\]
where $\|\cdot\|_p$ is the standard $\ell_p$-norm on $\R^n$. It follows that 
\[
\dstm(\Delta_n,\ast) = \inf_{\vec{\alpha}} \frac{1}{n^{1/p}} \|\vec{\alpha}\|_p,
\]
where the entries of $\vec{\alpha}$ are subject to constraints which guarantee a valid metric. Namely, we must have $\alpha_i > 0$ for all $i$, as well as  $\alpha_i - \alpha_j \leq 1 = d_n(i,j) \leq \alpha_i + \alpha_j$ for all $i \neq j$. In particular, the constraints of the form $1 \leq \alpha_i + \alpha_j$ with $i < j$ imply that
\[
\binom{n}{2} = \sum_{i < j} 1 \leq \sum_{i < j} (\alpha_i + \alpha_j) = (n-1) \sum_i \alpha_i = (n-1)\|\vec{\alpha}\|_1,
\]
whence we obtain, for all $p \geq 1$,
\[
\frac{1}{2} \leq \frac{1}{n} \|\vec{\alpha}\|_1 \leq \frac{1}{n^{1/p}} \|\vec{\alpha}\|_p,
\]
where the first inequality follows from the discussion above and the second is the generalized means inequality. Therefore, $\dstm(\Delta_n,\ast) \geq \frac{1}{2}$ holds for all $p \geq 1$. On the other hand, the constant vector $\vec{\alpha} = (1/2,\ldots,1/2)$ gives a valid metric and realizes this lower bound, so that $\dstm(\Delta_n,\ast) = \frac{1}{2}$.
\end{example}

\begin{remark}
    Example \ref{exmp:tightness} also shows that the inequality $\frac{1}{2} \dgw(\mX,\mY) \leq \dst(\mX,\mY)$ from \cite[Theorem 5.1(g)]{memoli2011gromov} (note the slightly different normalization convention used there) is tight, since it is not hard to see that $\dgw(\Delta_n,\ast) = \dgm(\Delta_n,\ast)$ and $\dst(\Delta_n,\ast) = \dstm(\Delta_n,\ast)$. It was previously only shown in \cite[Remark 5.14]{memoli2011gromov} that this is not an equality, in general.
\end{remark}

\subsection{Gromov-Monge Distances for Euclidean Spaces}\label{sec:GM_for_Euclidean_spaces}

A \define{Euclidean mm-space} is an mm-space $\mX$ such that $X \subset \R^n$ for some $n$ and $d_X$ is the restriction of Euclidean distance (and $\mu_X$ is a Borel measure with respect to the subspace topology). For Euclidean mm-spaces $\mX$ and $\mY$ with the same ambient space $\R^n$, we can consider the \define{isometry-invariant Monge $p$-distance}

\begin{linenomath}\begin{equation*}
\mathrm{M}^{\R^n, \mathrm{iso}}_{p}(\mX,\mY) = \inf_{T \in E(n)} \left(\inf_{\phi \in \mT(\mu_X, \mu_Y)} \int_{X} \|T(x) - \phi(x)\|^p \mu_X (dx)\right)^{1/p},
\end{equation*}\end{linenomath}
where we use $E(n)$ to denote the group of Euclidean isometries of $\R^n$ and $\|\cdot\|$ for the Euclidean norm. Fixing $\R^n$ as the ambient Euclidean space, unless specified otherwise, we use the cleaner notation $\mathrm{M}^\mathrm{iso}_p = \mathrm{M}^{\R^n, \mathrm{iso}}_{p}$.

\begin{theorem}\label{thm:gromov_monge_Euclidean}
Let $\mX$ and $\mY$ be Euclidean mm-spaces in the same ambient space $\R^n$, let $p \geq 1$ and let $M = \max\{\mathrm{diam}(X),\mathrm{diam}(Y)\}$. Then 
\begin{linenomath}\begin{equation*}
\dstm(\mX,\mY) \leq \mathrm{M}^\mathrm{iso}_p(\mX,\mY) \leq M^{3/4} \cdot c_n \cdot \big(\dgm^{\mathrm{em}}(\mX,\mY)\big)^{1/4},
\end{equation*}\end{linenomath}
where $c_n$ is a constant depending only on $n$. 
\end{theorem}

The theorem is analogous to \cite[Theorem 4]{memoli2008gromov}, which compares $\mathrm{GW}^{\mathrm{em}}_p$ and isometry-invariant Wasserstein distance for Euclidean spaces. The proof of the theorem is essentially the same, with the main difference being the use of the following technical lemma, which is a variation of \cite[Lemma 2]{memoli2008gromov}, specialized to the Monge setting. Recall that a \define{correspondence} between sets $X$ and $Y$ is a subset $R \subset X \times Y$ such that the coordinate projections satisfy $\rho_X(R) = X$ and $\rho_Y(R) = Y$. We use $\mathcal{R}(X,Y)$ to denote the set of correspondences between $X$ and $Y$. Observe that the theorem is only interesting if the quantity $M$ is finite, so we state the following result with this assumption.

\begin{lemma}\label{lem:euclidean_lemma}
Let $\mX$ and $\mY$ be  bounded mm-spaces and let $M = \max \{\mathrm{diam}(X),\mathrm{diam}(Y)\}$. If 
\begin{linenomath}\begin{equation*}
\dstm(\mX,\mY) \leq \epsilon \cdot M\end{equation*}\end{linenomath} 
for $\epsilon \leq 1$, then there exist $X_\epsilon \subset X$, $Y_\epsilon \subset Y$, $R_\epsilon \in \mathcal{R}(X_\epsilon,Y_\epsilon)$, and a measure-preserving map $\phi:X \rightarrow Y$ satisfying $\phi(X_\epsilon) = Y_\epsilon$ such that \begin{equation}\label{eqn:measure_condition}
\min\{\mu_X(X_\epsilon),\mu_Y(Y_\epsilon),\pi_{\phi}(R_\epsilon)\} \geq 1-\epsilon^{p/2}
\end{equation}
and
\begin{equation}\label{eqn:GH_condition}
\sup_{(x,y),(x',y') \in R_\epsilon} |d_X(x,x') - d_Y(y,y')| \leq \epsilon^{1/2}M.
\end{equation}
\end{lemma}

\begin{proof}
Without loss of generality, suppose that $X$ and $Y$ are subsets of an ambient metric space $(Z,d_Z)$ and let $\phi \in \mathcal{T}(\mu_X,\mu_Y)$ such that 
\begin{linenomath}\begin{equation*}
\int_Z d_Z(z, \phi(z))^p \mu_X(dz) \leq \epsilon^p M^p.
\end{equation*}\end{linenomath}
Define a subset $R_\epsilon \subset X \times Y$ by 
\begin{linenomath}\begin{equation*}
R_\epsilon = \{(x,y) \in X \times Y \mid y = \phi(x) \mbox{ and } d_Z(x,y) \leq \epsilon^{1/2} M /2\}.
\end{equation*}\end{linenomath}
We then define $X_\epsilon = \rho_X(R_\epsilon)$ and $Y_\epsilon = \rho_Y(R_\epsilon)$, so that $R_\epsilon \in \mathcal{R}(X_\epsilon,Y_\epsilon)$. Also note that $\phi(X_\epsilon) = Y_\epsilon$. A short calculation using the triangle inequality shows that  $R_\epsilon$ satisfies \eqref{eqn:GH_condition}. Moreover,
\begin{linenomath}\begin{equation*}
\epsilon^p M^p \geq \int_Z d_Z(z, \phi(z))^p \mu_X(dz) \geq \int_{X \setminus X_\epsilon} d_Z(z,\phi(z))^p \mu_X(dz) \geq \epsilon^{p/2} M^p \mu_X(X \setminus X_\epsilon).
\end{equation*}\end{linenomath}
Rearranging, we obtain $\mu_X(X_\epsilon) \geq 1- \epsilon^{p/2}$. Since $Y_\epsilon = \phi(X_\epsilon)$ and $\phi$ is measure-preserving, $\mu_Y(Y_\epsilon) \geq 1-\epsilon^{p/2}$ as well. Finally, we have
\begin{linenomath}\begin{equation*}
\pi_\phi(R_\epsilon) = (\mathrm{id}_X \times \phi)_\# \mu_X (R_\epsilon) = \mu_X\big((\mathrm{id}_X \times \phi)^{-1}(R_\epsilon)\big) = \mu_X(X_\epsilon),
\end{equation*}\end{linenomath}
and \eqref{eqn:measure_condition} is satisfied.
\end{proof}

\begin{proof}[Proof of Theorem \ref{thm:gromov_monge_Euclidean}]
The inequality on the left is obvious, so let us consider the inequality on the right, under the assumption that $\mX$ and $\mY$ are bounded (as this is the only interesting case). If either quantity is infinite, then both are, so assume not. It is easy to show that 
\begin{equation*}
\dgm^{\mathrm{em}}(\mX,\mY) \leq M = \max\{\mathrm{diam}(X),\mathrm{diam}(Y)\}.
\end{equation*}
We therefore suppose, without loss of generality, that $\dgm^{\mathrm{em}}(\mX,\mY) = \epsilon M$ for some $\epsilon \leq 1$. Let $X_\epsilon$, $Y_\epsilon$, $\phi$ and $R_\epsilon$ be as in Lemma \ref{lem:euclidean_lemma} and let $R_\epsilon^c = (X \times Y) \setminus R_\epsilon$. The condition \eqref{eqn:GH_condition} on $R_\epsilon$ implies that the Gromov-Hausdorff distance between $X_\epsilon$ and $Y_\epsilon$ is bounded above by $\epsilon^{1/2}M/2$ (see \cite[7.3.25]{burago2001course}). In turn, \cite[Corollary 7.3.28]{burago2001course} says that this implies that there exists a \emph{$(\epsilon^{1/2}M)$-isometry} $\psi:X_\epsilon \to Y_\epsilon$; that is, $\psi$ satisfies
\begin{linenomath}\begin{equation*}
\sup_{x,x' \in X_\epsilon} |d_X(x,x') - d_Y(\psi(x),\psi(x')) | \leq \epsilon^{1/2} M
\end{equation*}\end{linenomath}
and that $\psi(X_\epsilon)$ is a $( \epsilon^{1/2}M)$-net for $Y_\epsilon$ (i.e., for every $y\in Y_\epsilon$ there is an $x \in X_\epsilon$ such that $d_Y(\psi(x),y) \leq \epsilon^{1/2} M$). We apply \cite[Theorem 2.2]{alestalo2001isometric} to conclude that there is an isometry $T \in E(n)$ such that $\sup_{x \in X_\epsilon} \|T(x) - \psi(x)\| \leq \epsilon^{1/4} \cdot a_n \cdot M$, where $a_n$ is a constant depending only on $n$. 

Applying the triangle inequality and the general inequality $(a+b)^p \leq 2^{p-1}(a^p + b^p)$ (for $a,b \geq 0$ and $p \geq 1$), we have
\begin{linenomath}\begin{align}
\mathrm{M}^\mathrm{iso}_p(\mX,\mY)^p &\leq  \int_{R_\epsilon^c}\|T(x) - y\|^p \pi_\phi (dx \times dy)  \label{eqn:monge_bound_1} \\
&\hspace{.2in} + 2^{p-1} \int_{R_\epsilon} \|T(x) - \psi(x)\|^p \pi_\phi(dx \times dy)  \label{eqn:monge_bound_2}  \\
&\hspace{.2in} + 2^{p-1} \int_{R_\epsilon} \|\psi(x) - y\|^p \pi_\phi(dx \times dy). \label{eqn:monge_bound_3}
\end{align}\end{linenomath}
We bound each term separately. First note that 
\begin{linenomath}\begin{equation*}
\|T(x) - y \|^p \leq 2^{p-1} (\|T(x)\|^p + \|y\|^p) \leq 2^{p} \max \left\{\max_x \|T(x)\|^p,\max_y \|y\|^p\right\} = (2M)^p,
\end{equation*}\end{linenomath}
where we have used isometry invariance of $\mathrm{M}^\mathrm{iso}_p$ to assume without loss of generality that the circumcenters of $X$ and $Y$ are at the origin in order to obtain the last equality. This implies that
\begin{linenomath}\begin{equation*}
\eqref{eqn:monge_bound_1} \leq (2M)^p \pi_\phi(R_\epsilon^c) \leq 2^p \cdot  M^p \cdot \epsilon^{p/2}. 
\end{equation*}\end{linenomath}
The bounds on \eqref{eqn:monge_bound_2} and \eqref{eqn:monge_bound_3} follow from our assumptions on $\psi$:
\begin{linenomath}\begin{equation*}
\eqref{eqn:monge_bound_2} \leq 2^{p-1} \cdot \sup_{x \in X_\epsilon} \|T(x) - \psi(x) \|^p \cdot  \pi_\phi(R_\epsilon) \leq 2^{p-1} \cdot \epsilon^{p/4} \cdot a_n^p \cdot M^p 
\end{equation*}\end{linenomath}
and
\begin{linenomath}\begin{equation*}
\eqref{eqn:monge_bound_3} \leq 2^{p-1} \cdot \max_{(x,y) \in R_\epsilon} \|\psi(x) - y\|^p \pi_\phi(R_\epsilon) \leq 2^{p-1} \cdot \epsilon^{p/2} \cdot M^p,
\end{equation*}\end{linenomath}
where we have used that $\psi(X_\epsilon)$ is a $(\epsilon^{1/2}M)$-net. Combining these estimates, we conclude 
\begin{linenomath}\begin{equation*}
\mathrm{M}^\mathrm{iso}_p(\mX,\mY)^p \leq 2^p M^p (\epsilon^{p/2} + a_n^p \epsilon^{p/4}/2 + \epsilon^{p/2}/2) \leq M^p \epsilon^{p/4} \cdot 2^p (3/2 + a_n^p/2).
\end{equation*}\end{linenomath}
Taking $c_n = 3 + 2a_n \geq  2 \cdot (3/2 + a_n^p/2)^{1/p}$ (to get a constant which only depends on $n$), we have 
\[
\mathrm{M}^\mathrm{iso}_p(\mX,\mY) \leq c_n \cdot M \cdot \epsilon^{1/4} = c_n \cdot M^{3/4} \cdot \dgm^\mathrm{em}(\mX,\mY)^{1/4}.
\]
\end{proof}

\begin{remark}\label{rmk:more_examples}
    We now explain connections between embedded GM distances and some metrics which have previously appeared in the literature.
    
    The \define{continuous Procrustes distance} between embedded surfaces in $\R^3$ was introduced in \cite{boyer2011algorithms}, where it was used to classify anatomical surfaces (in particular, shapes of primate teeth). The effectiveness at classification of this metric was shown to be roughly on par with that of a trained morphologist. The idea of the continuous Procrustes distance is to compare surfaces by simultaneously registering over rigid motions while looking for optimal measure-preserving maps between them.  Theoretical aspects of this metric are studied in \cite{al2013continuous}, where it is shown that optimal mappings are close to being conformal. The continuous Procrustes distance can be viewed as $\mathrm{M}^{\R^3,\mathrm{iso}}_2$, under the additional constraint that measure-preserving maps must also be continuous. A similar constraint on admissible maps can be made in the $\mathrm{GM}^{\mathrm{em}}_p$-distance. Theorem \ref{thm:gromov_monge_Euclidean} (or the corresponding version under an additional constraint) gives an equivalence between embedded GM distance and continuous Procrustes distance.

    A metric similar to the continuous Procrustes metric is studied in \cite{haker2004optimal} for applications to 2D image registration. In our terminology, the metric of \cite{haker2004optimal} is $\mathrm{M}^{\R^2,\mathrm{iso}}_2$, under the additional constraint that measure-preserving maps are smooth diffeomorphisms. Once again, (a slight variant) of Theorem \ref{thm:gromov_monge_Euclidean} shows the equivalence of this metric with a restricted version of embedded GM distance.
\end{remark}

\begin{remark}
  A natural question is whether the inequalities in Theorem \ref{thm:gromov_monge_Euclidean} are tight; in particular, can the exponent on the right hand side be improved? There has been some progress on the corresponding problem comparing isometry-invariant Hausdorff distance with Gromov-Hausdorff distance between Euclidean sets---it is shown in \cite[Theorem 2]{memoli2008gromov} that the isometry-invariant Hausdorff distance between compact subsets $X,Y \subset \R^n$ is bounded above by $c_n' \cdot \max\{\mathrm{diam}(X),\mathrm{diam}(Y)\}^{1/2} \cdot \mathrm{GH}(X,Y)^{1/2}$, where $\mathrm{GH}$ denotes Gromov-Hausdorff distance and $c_n'$ is positive a constant depending only on $n$. An example constructed in \cite{memoli2008gromov} shows that the exponent $\tfrac{1}{2}$ on the Gromov-Hausdorff distance cannot   be improved (namely, it cannot be made larger).  More recently, it has been shown that for finite subsets of the real line $X,Y \subset \R$, isometry-invariant Hausdorff distance is bounded above by $\frac{5}{4} \mathrm{GH}(X,Y)$~\cite[Theorem 3.2]{majhi2024approximating}, and that this bound is tight~\cite[Theorem 3.10]{majhi2024approximating}. Optimality of the bounds in the Gromov-Wasserstein or Gromov-Monge settings has seen less progress, and we leave these as open questions.
\end{remark}

\subsubsection*{Acknowledgements}

We acknowledge funding from these NSF projects:  DMS 2107808, DMS 2324962,  RI 1901360, CCF 1740761,  CCF 1526513, and DMS 1723003 and also from BSF project 2020124.

\bibliography{needham_bibliography}

\end{document}